\documentclass{amsproc}

\usepackage{amssymb}

\usepackage{color}
\usepackage{hyperref}

\newtheorem{theorem}{Theorem}[section]
\newtheorem{lemma}[theorem]{Lemma}

\newtheorem{proposition}{Proposition}[section]

\newtheorem{corollary}[theorem]{Corollary}
\theoremstyle{definition}

\theoremstyle{remark}
\newtheorem{remark}[theorem]{Remark}

\numberwithin{equation}{section}

\newcommand{\bbC}{\mathbb{C}}
\newcommand{\bbR}{\mathbb{R}}

\newcommand{\bbN}{\mathbb{N}}

\newcommand{\bbZ}{\mathbb{Z}}

\newcommand{\cE}{\mathcal{E}}
\newcommand{\cP}{\mathcal{P}}

\newcommand{\cH}{\mathcal{H}}
\newcommand{\cS}{\mathcal{S}}
\newcommand{\cU}{\mathcal{U}}
\newcommand{\cV}{\mathcal{V}}
\newcommand{\cZ}{\mathcal{Z}}

\newcommand{\ev}{\mathop{\mathrm{ev}}\nolimits}

\newcommand{\card}{\mathop{\mathrm{card}}}

\newcommand{\SO}{\mathop{\mathrm{SO}}}

\newcommand{\so}{\mathop{\mathrm{so}}}

\newcommand{\one}{{\mathord{\mathbf1}}}

\newcommand{\al}{{\mathord{\alpha}}}
\newcommand{\be}{{\mathord{\beta}}}
\newcommand{\ga}{{\mathord{\gamma}}}
\newcommand{\Ga}{{\mathord{\Gamma}}}
\newcommand{\vf}{{\mathord{\varphi}}}

\newcommand{\si}{{\mathord{\sigma}}}

\newcommand{\vk}{{\mathord{\varkappa}}}

\newcommand{\la}{{\mathord{\lambda}}}

\newcommand{\De}{{\mathord{\Delta}}}

\newcommand{\ep}{{\mathord{\varepsilon}}}

\newcommand{\ze}{{\mathord{\zeta}}}

\newcommand{\frg}{\mathord{\mathfrak{g}}}
\newcommand{\frh}{{\mathord{\mathfrak{h}}}}

\newcommand{\scal}[2]{\left<#1,#2\right>}
\newcommand{\dist}{\mathop{\mathrm{dist}}\nolimits}

\newcommand{\sfE}{{\mathsf{E}}}

\let\td=\tilde
\let\wtd=\widetilde

\begin{document}

\title[Decomposition of the Kostlan--Shub--Smale model]
{Decomposition of the Kostlan--Shub--Smale model for random polynomials}


\author{V. Gichev}
\address{Sobolev Institute of Mathematics\\
Omsk Branch\\
ul. Pevtsova, 13, \\ 644099, Omsk, Russia}
\email{gichev@ofim.oscsbras.ru}
\curraddr{}
\thanks{The author was partially supported by
the grant of the Norwegian Research Council \#204726/V30}



\date{}

\begin{abstract}
Let $\cP_n$ be the space of homogeneous polynomials of degree $n$ on $\bbR^{m+1}$.
We consider the asymptotic behavior of some coefficients relating to the decomposition
of $\cP_n$ into the sum of $\SO(m+1)$-irreducible components. Using the results,
we prove that a random Kostlan--Shub--Smale polynomial $u\in\cP_n$ can be
approximated by polynomials of lower degree in the Sobolev spaces $H^k(S^m)$
on the unit sphere $S^m$ with small error and probability close to $1$. For example, if
$l_n>\sqrt{(m+2k+8\ep)n\ln n}$, then the inequality
$\dist(u,\cP_{l_n})<An^{-\ep}\|u\|$ holds for any sufficiently large $n$ with probability
greater than $1-Bn^{-2\ep}$, where $\dist$ and $\|\ \|$ are the distance and norm in
$H^k(S^m)$,  respectively,  $\ep\in(0,1)$,
and $A,B$ depend only on $m$ and $k$. If $l_n>\ep n$, then both the approximation
error and the deviation of probability from $1$ decay exponentially.

\end{abstract}

\maketitle


\section{Introduction}

Let $G$ be a compact Lie group acting on a Riemannian manifold $M$
by isometries, $\cE$ be a finite dimensional $G$-invariant
subspace of $C^\infty(M)$. For $u\in{\cE}$, set
$N_u=u^{-1}(0)$. Let $\sigma$ be a $G$-invariant probability
measure on $\cE$. Then the Hausdorff measure $\frh^{m-1}(N_u)$ of
$N_u$, where $m=\dim M$, is a random variable whose distribution
depends on $\sigma$. There are other metric quantities which can
be considered in this setting, for example, the Euler
characteristic of $N_u$ (see \cite{Po99}, \cite{Bu07}). Most of
the known results were proved for the Gaussian distributions in
$\cE$ and the uniform distribution in the unit sphere in $\cE$.

We consider the case $G=\SO(m+1)$, $M=S^m=\SO(m+1)/\SO(m)$,
the unit sphere in $\bbR^{m+1}$, and $\cE=\cP_n$, the space of real
homogeneous of degree $n$ polynomials on $\bbR^{m+1}$. Let $\cH_j$
be the subspace of all harmonic polynomials in $\cP_j$. There is
the well known
 $\SO(m+1)$-invariant decomposition
\begin{eqnarray}\label{harmdec}
\cP_n=\sum_{j\in J_n}|x|^{n-j}\cH_{j},
\\
\label{defjnz}
J_n=\{j\in\bbZ:\,0\le j\leq n,~n-j~\,\text{even}\}.
\end{eqnarray}
The restriction onto $S^m$ is an injective
map on the space $\cP_n$. The space $\sum_{j\in J_n}\cH_j$ has the
same property since it consists of harmonic polynomials. By (\ref{harmdec}),
it has the same traces on $S^m$.
We denote by $\cP_n$, $\cH_j$ the function
spaces on $S^m$ as well as on
$\bbR^{m+1}$ hoping that no confusion will occur. Due to this
convention,
\begin{eqnarray*}
\cP_j\subseteq\cP_n~~~~\hbox{\rm if}~~j\in J_n.
\end{eqnarray*}
The investigation of random polynomials was initiated by papers
\cite{BP} by Bloch and Polya, and \cite{LO38}, \cite{LO39} by
Littlewood and Offord. In \cite{Kac}, M.\,Kac proved an exact
integral formula for the expectation of the number of real zeroes
of random polynomials of one variable with standard Gaussian
coefficients (i.e., having expectation $0$ and variance $1$).
Kostlan (see \cite{Ko93}, \cite{EK95}) found a geometric proof of
this formula. For any $a\in\bbR^{n+1}$, common points of the
hyperplane $a^\bot$ and the moment curve $\ga(x)=(1,x,\dots,x^n)$
are in one-to-one correspondence with zeroes of the polynomial
$\scal{a}{\ga(x)}$. The same is true for the central projection
$\td\ga(x)=\frac{\ga(x)}{|\ga(x)|}$ of $\ga$ onto the unit sphere
$S^n$. Set  $f(a)=\card(a^\bot\mathop\cap\td\ga)$. Since $f(a)$ is
homogeneous of degree $0$, we can compute the expectation of the
number of zeroes integrating $f$ over $S^n$; on the other hand,
the integral is proportional to the length of $\td\ga$ due to a
Crofton type formula.

This method can be extended to the other function spaces and
inner products. Kostlan noted that the distribution on the space
of polynomials whose coefficients are independent Gaussian with
the variance ${{n}\choose{j}}$ at $x^j$ has a hidden symmetry: it
can be lifted onto the space of homogeneous degree $n$ polynomials
of two variables as an $\SO(3)$-invariant distribution. In
\cite{Ko93}, Kostlan found the expectation of the volume of 
solutions to a random system of equations $u_j(x)=0$, where
$u_j\in\cP_n$ are independent random polynomials, $j=1,\dots,k$,
where $k\leq m$. It is proportional to $n^{\frac{k}{2}}$. If $k=m$,
then it is the mean number of solutions and is equal to $n^{\frac{m}{2}}$. 
Shub and Smale in the paper \cite{SS93} extended this result onto the case 
of different degrees: the expectation of the number of solutions to the system
$u_1(x)=\dots=u_m(x)=0$, $u_j\in\cP_{n_j}$, is equal to 
$\sqrt{n_1\dots n_m}$.

In the paper \cite{Po99}, Podkorytov introduced a parameter of a
Gaussian $\SO(m)$-invariant distribution in $\cP_n$ and found an
explicit formula for the expectation of the Euler characteristic
of $N_u$ which is a function of this parameter.
{
In higher codimensions}
(i.e., for the varieties
$N_{u_1}\mathop\cap\dots\mathop\cap N_{u_k}$, where $k\leq m$),
B\"urgisser computed the expectation of the Euler characteristic
in the paper \cite{Bu07}. His proof involves Weyl's tube formula.

The space $M=G/H$ is called isotropy irreducible if $H$ is irreducible in $T_oM$, where
$o$ is the base point of $M$ corresponding to $H$. Such a space admits the
unique up to a scaling factor invariant Riemannian metric. Hence any equivariant
non-constant mapping of $M$ into a Riemannian $G$-manifold is a
finite covering and a local metric homothety onto its image. We assume $M$
 isotropy irreducible in what follows.

The $G$-invariant inner product in $\cE$ defines two parameters
relating to the geometry of the evaluation mapping $\ev_{\cE}:\,M\to\cE$ (see Section~2
for precise definitions). We denote them by $c$ and $s$ throughout the paper.
The mapping $\ev$ is an immersion of  $M$ into the sphere of radius $c$:
\begin{eqnarray}\label{defcevp}
c=|\ev_{\cE}(p)|,~~p\in M.
\end{eqnarray}
Thus, $\iota=\frac1c\ev_{\cE}$ is a non-constant
mapping  into the unit sphere $\cS$ in $\cE$.  Let $s$ denote the coefficient of
local metric homothety for $\iota$:
\begin{eqnarray}\label{defss}
s=\frac{|d_p\iota(v)|_{\cE}}{|v|_{T_pM}},
\end{eqnarray}
This parameter was  introduced in \cite{Gi08}. For Gaussian distributions,
$s^2$ is the Podkorytov parameter.

The parameters $c$ and $s$ are essential ingredients in the formulas for the
averages and fluctuations of some geometric quantities.  Here is an example. Locally,
the mapping $\iota$ multiplies the $k$-dimensional Hausdorff measure $\frh^{k}$
on $s^k$.   This makes it possible to compute or estimate Hausdorff measure of
a set $X$ in $M$ applying  Federer's kinematic formula for spheres  (\cite[Theorem~3.2.48]{{Fe69}}) to $\iota(X)$. Let $X=N_{u}$, $u\in\cS$.
Integrating over $\cS$ we get  $\sfE\big(\frh^{m-1}(N_u)\big)=
\frac{\varpi\varpi_{m-1}}{\varpi_{m}}s$, where $\varpi$ and
$\varpi_k$ are volumes of $M$ and the unit sphere $S^k$ in
$\bbR^{k+1}$, respectively. For the expectation of measures of the
intersections of the sets $N_u$, there is a similar expression
with the product of the coefficients of metric homothety.

Sometimes, it is possible to find $s$ following the
definition (i.e., applying (\ref{defss})). This is the case in the
Kostlan--Shub--Smale model (see Section~\ref{seccoepn} for
the definition). If $\cE$ is an eigenspace of the
Laplace--Beltrami operator $\De_{M}$ and $\la$ is the eigenvalue of $-\De_{M}$,
then
\begin{eqnarray}\label{sbylam}
s=\sqrt{\frac{\la}{m}}
\end{eqnarray}
independently of the choice of the invariant inner product.
Since $M$ is isotropy irreducible,
any $G$-irreducible component of $\cE$ is an eigenspace of $\De_M$. In
general, $s$ depends on the irreducible components of $\cE$ and
the inner product in $\cE$. For the norm of $L^2(M)$ in $\cE$  the
answer is given in \cite[Lemma~1]{Gi13}: if
\begin{eqnarray}\label{decce}
\cE=\cE_1\oplus\dots\oplus\cE_l,
\end{eqnarray}
where the summands are $G$-irreducible eigenspaces of $\De_{M}$, then
\begin{eqnarray}\label{ssqunu}
s^2=\nu_1s_1^2+\dots+\nu_ls_l^2,
\end{eqnarray}
where $s_j$ is the coefficient of metric homothety for $\cE_j$,
which is subject to (\ref{sbylam}), and
$\nu_j=\frac{\dim\cE_j}{\dim\cE}$, $j=1,\dots,l$.
{In Section~\ref{seccoef}, we show that (\ref{ssqunu}) holds for
any invariant inner product in $\cE$ and $\nu_j=\frac{c_{j}^{2}}{c^{2}}$,
where $c_{j}$ and $c$ are the parameters for $\cE_{j}$ and $\cE$, respectively.
The coefficients $\nu_{j}$ depend on the choice of the inner product but $s_j$
are always subject to (\ref{sbylam}).}  This makes it is possible to find $c$ and
$s$ for the rescaling inner products in $\cE$. Section~2 contains
the preparatory material which can be used if $\cE$ is a finite dimensional
$G$-invariant function space with arbitrary spectrum on any isotropy irreducible
homogeneous space $M$. 

In the case $M=S^m$,  $\cE=\cP_n$, and the decomposition (\ref{harmdec}),
$s\sim\frac{n}{\sqrt{m+2}}$ as $n\to\infty$ for the $L^2(S^m)$-norm
in $\cE$. In the Kostlan--Shub--Smale model, $s=\sqrt{n}$ independently of $m$.
In Proposition~\ref{kss-l2}, we collect the formulas for the parameters $c$, $s$
in the spaces $\cP_{n}$ and $\cH_{j}$,  the coefficients $\nu_{j}$ and the rescaling
factors for the $L^{2}(S^{m})$ and  Kostlan--Shub--Smale ensembles. In the recent
paper \cite{FLL15}, Fyodorov, Lerario, and Lundberg found, among other results,
these factors and their scaling limit  as $n\to\infty$ which is equal to $e^{-\frac{x^{2}}{4}}$
up to a normalizing factor.

The coefficients $\nu_{j}$ have a peak near $\sqrt{(m-1)n}$ and
decay very fast when $j$ grows.
The peak can be localized in an interval of
length $\frac{m+5}{2}$ and is sharpening as $m$ grows.
In Theorem~\ref{limitnu}, we compute the scaling limit of $\nu_{j}$. Up to
a normalizing factor, it is equal to $\left(t^{2}e^{{1-t^{2}}}\right)^{{\frac{m-1}{2}}}$.

The main result of the paper
is Theorem~\ref{lower}. It implies that
a random Kostlan--Shub--Smale polynomial of degree $n$ admits a
good approximation in the Sobolev spaces by polynomials of
degree less than $\sqrt{(m+2k+\ep)n\ln n}$ with high
probability. For example, if $k=0$, $n$ is sufficiently large,
\begin{eqnarray*}
n>l_n>2\sqrt{mn\ln n},
\end{eqnarray*}
and $n-l_n$ is even, then the inequality
\begin{eqnarray*}
\dist\left(u,\cP_{l_n}\right)<\ep_n|u|
\end{eqnarray*}
holds with the probability greater than $1-\eta_n$, where
\begin{eqnarray*}
\ep_n=a n^{-\frac{m}{2}},~~~\eta_n=bn^{-\frac{m}{2}},
\end{eqnarray*}
$\dist$ stands for the distance in $L^2(S^m)$, and $a,b$ are
independent of $n$. If $l_n>\al n$, where $0<\al<1$, then $\ep_n$
and $\eta_n$ may both decay exponentially when $n$ grows.
Theorem~\ref{lower} provides .
This is not clear yet if $\sqrt{(m+2k)n\ln n}$ is actually
the critical rate of growth for the approximation in $H^{k}(S^{m})$
by polynomials of lower degree.
{
Since $\dim\cP_{n}$ grows as $n^{m}$ as $n\to\infty$, the measure
concentration phenomenon works in this situation. Hence the
inequality above holds with high probability due to the fast decay of
the coefficients $\nu_{j}$ .

Throughout the paper, we fix $m$ and drop it in the notation assuming
\begin{eqnarray}\label{assmn}
1<m<n.
\end{eqnarray}
We use the notation $|\ |$ for the Euclidean norms and $\scal{\
}{\ }$ for the corresponding inner product. The base point $o$ of $M=G/H$
is the class $H$. If $M$ is the unit sphere
$S^m$ in $\bbR^{m+1}$, then $o=(1,0,\dots,0)$. In the notation
$L^2(M)$, the invariant probability measure on $M$ is assumed.
Also, $du$, $dx$, etc. stands either for the Lebesgue measure in
an Euclidean space or for the invariant probability measure on a
compact homogeneous space (in particular, on $S^m$).

\subsubsection*{Acknowledgements} I am grateful to the unknown referee for
useful comments and for making me aware of the papers \cite{FLL15} and
\cite{Ko02}. A part of this work was done during my stay at the University of Bergen,
in the warm and friendly atmosphere created by Irina Markina and  Aleksandr (Sasha) 
Vasiliev.   
Sasha's untimely decease is an irreplaceable loss, he was a wonderful person and a gifted 
mathematician. This paper is dedicated to his memory.

\section{The coefficients corresponding to
invariant Euclidean structures}\label{seccoef}

Since $M$ is isotropy irreducible, the invariant Riemannian metric
on it is unique up to a scaling factor. Hence it is the quotient of
some bi-invariant metric on $G$. This implies that any $G$-invariant
finite dimensional function space on $M$ is $\De_M$-invariant (the
introduction in \cite{Gi13} contains more details; for
$M=S^m$ this is true because the summands in \ref{harmdec} are irreducible
and pairwise non-equivalent).
Hence the summands of the $G$-invariant orthogonal decomposition (\ref{decce})
are eigenspaces of $\De_{M}$.
Let $\la_j$ denote the eigenvalue of $-\De_{M}$ on $\cE_j$. We
assume $\cE_j\neq0$ for all $j=1,\dots,l$.
}

The evaluation mapping $\ev_{\cE}:\,M\to\cE$ is defined by the
identity
\begin{eqnarray}\label{defev}
\scal{\ev_{\cE}(p)}{u}=u(p),
\end{eqnarray}
where $p\in M$, $u\in\cE$. Then
$\iota(p)=\frac{\ev_{\cE}(p)}{|\ev_{\cE}(p)|}$.
Due to the homogeneity of $M$, $|\ev_{\cE}(p)|$ is independent of
$p$. Set
\begin{eqnarray}
\phi=\ev_{\cE}(o),\nonumber
\end{eqnarray}
and $c_j=|\ev_{\cE_j}(p)|$, $\phi_j=\ev_{\cE_j}(o)$, where
$j=1,\dots,l$. Clearly, $\phi=\phi_1+\dots+\phi_l$. Since $\phi_j$
are pairwise orthogonal, we get
\begin{eqnarray}\label{cjcjsq}
c^2=c_1^2+\dots+c_l^2,
\end{eqnarray}
where $c$ is defined by (\ref{defcevp})
Let $\frg$ be the Lie algebra of $G$.
\begin{lemma}\label{snusj}
Set $\nu_j=\frac{c_j^2}{c^2}$, and let $s_j$ be the coefficient of
the metric homothety for $\cE_j$, $j=1,\dots,l$. Then
\begin{eqnarray}\label{sumnul}
\nu_1+\dots+\nu_l=1,\\
\label{csnucjsj}
s^2=\nu_1s_1^2+\dots+\nu_ls_l^2.
\end{eqnarray}
\end{lemma}
\begin{proof} The equality (\ref{sumnul}) follows from (\ref{cjcjsq}).
For any $\xi\in\frg$, we have
$d_p\iota(\xi(p))=\frac{1}{c}\xi\ev_{\cE}(p)$ and, according to
(\ref{defss}),
\begin{eqnarray}\label{xiphi}
cs|\xi(o)|_{T_oM}=|\xi\phi|_{\cE}.
\end{eqnarray}
Similarly, $s_jc_j|\xi(o)|_{T_oM}=|\xi\phi_j|_{\cE_j}$. Since
$\xi\phi_j\in\cE_j$,
\begin{eqnarray*}
|\xi\phi|_{\cE}^2=|\xi\phi_1+\dots+\xi\phi_l|^2_{\cE}
=|\xi\phi_1|_{\cE_1}^2+\dots+|\xi\phi_l|_{\cE_l}^2
=(c_1^2s_1^2+\dots+c_l^2s_l^2)|\xi(o)|_{T_oM}.
\end{eqnarray*}
Together with (\ref{xiphi}), this implies (\ref{csnucjsj}).
\end{proof}

Let $\scal{\ }{\ }$, $\wtd{\scal{\ }{\ }}$ be $G$-invariant inner
products in $\cE$, $\cS$ and $\wtd\cS$ be the corresponding unit
spheres in $\cE$, respectively. We shall endow with the tilde the
notation for relating objects. We assume additionally that there
are $\tau_j>0$, $j=1,\dots,l$, such that for all $u,v\in\cE$
\begin{eqnarray}\label{nnorms}
\wtd{\scal{u}{v}}=\tau_1^{-1}\scal{u_1}{v_1}+
\dots+\tau_l^{-1}\scal{u_l}{v_l},
\end{eqnarray}
where $u_j,v_j$ are components of $u,v$ in the decomposition
(\ref{decce}). The assumption holds if the summands in
(\ref{decce}) are pairwise non-equivalent as $G$-modules. This is true
if $M=S^m$.
\begin{lemma}\label{eqisct}
For any $j=1,\dots,l$, the following equalities hold:
\begin{eqnarray*}
\begin{array}{r}
\td c_j^2=\tau_j c_j^2 ,\\
\td\phi_j=\tau_j\phi_j,\\
\td s_j=s_j=\sqrt{\frac{\la_j}{m}},
\end{array}
\end{eqnarray*}
and, moreover, $\td s^2=\frac{\tau_1 c_1^2}{\td c^2}s_1^2+
\dots+\frac{\tau_l c_l^2}{\td c^2}s_l^2$, where $\td c^2=\tau_1
c_1^2+\dots+\tau_l c_l^2$.
\end{lemma}
\begin{proof}
Suppose $\scal{u}{v}=\tau\wtd{\scal{u}{v}}$ for all $u,v\in\cE$.
Due to the equalities
$u(o)=\wtd{\scal{u}{\td\phi}}=\scal{u}{\frac1{\tau}\td\phi}
=\scal{u}{\phi}$, we have $\td\phi=\tau\phi$ and $\td
c^2=\wtd{|\td\phi|}^2=\tau|\phi|^2=\tau c^2$ in this case. By
(\ref{xiphi}),
\begin{eqnarray*}
cs|\xi(o)|_{T_oM}=|\xi\phi|,\\ 
c\td s|\xi(o)|_{T_oM}=\wtd{|\xi\td\phi|}\phantom{!}
\end{eqnarray*}
for all $\xi\in\frg$. The equality $\td\phi=\tau\phi$ implies
$\wtd{|\xi\td\phi|}^2=\tau|\xi\phi|^2$ and, together with $\td
c^2=\tau c^2$ and the equalities above,  $\td s=s$. According to
\cite[Lemma~1]{Gi13}, in the case of $L^2(M)$-norm and $\cE$
irreducible we have $s=\sqrt{\frac{\la}{m}}$, where $\la$ is the
eigenvalue of $-\De_M$ on $\cE$ (in \cite{Gi13}, it is assumed that
$\cE\perp\one$ but for the space of constant functions the
equality $s=\sqrt{\frac{\la}{m}}$ is evidently true since
$\la=s=0$). Since the invariant inner products on irreducible
$G$-modules are pairwise proportional, the arguments above prove
the first three equalities. The remaining ones follow from
Lemma~\ref{snusj} and (\ref{cjcjsq}).
\end{proof}
\begin{corollary}
Let $\la_{\min}$ and $\la_{\max}$ be the least and largest
eigenvalues of $-\De_M$ in $\cE$, respectively. Suppose $\la_{\min}<
\la_{\max}$.  Then
\begin{eqnarray*}
\sqrt{\frac{\la_{\min}}{m}}< s<\sqrt{\frac{\la_{\max}}{m}}.
\end{eqnarray*}
Moreover, for any $s$ satisfying this inequality there is an
invariant Euclidean norm on $\cE$ whose coefficient of metric
homothety is equal to $s$.
\end{corollary}
\begin{proof}
{
If $\cE_j$ contains a non-constant function, then $\la_j>0$ and, consequently, $\nu_j>0$.} Thus the equality (\ref{sumnul})
implies the inequalities above. Since $\tau_j$'s may be arbitrary
positive numbers, the second assertion of the corollary follows
from Lemma~\ref{eqisct}.
\end{proof}
The equivalent inequalities for $\SO(m+1)$-invariant Gaussian distributions
$\cP_n$ were stated without proof  in Podkorytov's paper  \cite{Po99}.


\section{The coefficients for the space of homogeneous polynomials}
\label{seccoepn} We refer to \cite[Chapter~5]{Ax01} for known
facts on harmonic polynomials. In this section, we find the
coefficients introduced in the previous one for the decomposition
(\ref{harmdec}) and the following Euclidean norms: the first, $|\
|$, is the norm of $L^2(S^m)$ for the invariant probability
measure and the second, $\wtd{|\ |}$, is defined by
\begin{eqnarray}\label{decph}
\wtd{\scal{x^\al}{x^\be}}=\begin{cases}\al!,&\al=\be
\\0,&\al\neq\be,
\end{cases}
\end{eqnarray}
where $\al=(\al_0,\al_1,\dots,\al_m)$, $\al!=\al_0!\dots\al_m!$,
and $x^\al=x_0^{\al_0}\dots x_m^{\al_m}$.
A short straightforward computation shows that the formula
\begin{eqnarray}\label{sympro}
\wtd{\scal{u}{v}}=u\left(\frac{\partial}{\partial x}\right)v
\end{eqnarray}
defines the same inner product in $\cP_n$ (notice that the
right-hand part is constant). To the best of my knowledge, it was
introduced in the the book \cite{St70} by E.\,Stein. The products
are $\SO(m+1)$-invariant. The Kostlan--Shub--Smale model
corresponds to the Gaussian distribution whose density is
proportional to $e^{-\wtd{|u|}^2}$ on $\cP_n$.

Let us use the notation of (\ref{harmdec}). On $S^m$, $\cH_j$ is the
eigenspace of the Laplace--Beltrami operator  $\De_{S^{m}}$
corresponding to the eigenvalue $-j(j+m-1)$. Set
$\la_j=j(j+m-1)$.  It is known that
$\dim\cP_n={{n+m}\choose{m}}$.
According to (\ref{harmdec}),
\begin{eqnarray}\label{dimhn}
\dim\cH_n={{n+m}\choose{m}}-{{n+m-2}\choose{m}}
=\frac{(m+n-2)!(m+2n-1)}{(m-1)!n!}
\end{eqnarray}
if $n\geq2$. Clearly, $\dim\cH_0=1$, $\dim\cH_1=m+1$ and,
moreover, the equalities hold for $n=0,1$ if we replace the
factorials with $\Ga$ and extend the right-hand side of
(\ref{dimhn}) analytically. For short, set
\begin{eqnarray*}
\vk(x)=|x|^2=x_0^2+x_1^2+\dots+x_m^2.
\end{eqnarray*}
\begin{lemma}\label{taukj}
Let $u\in\cH_j\setminus\{0\}$ and $k\geq0$ be integer. Then
\begin{eqnarray}\label{rescal}
\frac{\wtd{|\vk^k u|}^2}{|u|^2}= 2^k k!\prod_{i=1}^{j+k}(m+2i-1)
\end{eqnarray}
\end{lemma}
\begin{proof}
For $u\in\cP_j$, let $u(x)=\sum_{|\al|=j}u_\al x^\al$, where
$|\al|=\al_0+\dots+\al_m$ and $x^{\al}=x_{0}^{\al_{0}}\cdots x_{m}^{\al_{m}}$,
be its decomposition into the sum of monomials. By definition,
\begin{eqnarray}\label{kosnorm}
\wtd{|u|}^2=\sum_{|\al|=j}u_\al^2\al!,
\end{eqnarray}
If $u\in\cH_j$, then its $L^2$-norm $|u|$ can be computed by a
similar formula (\cite[Theorem~5.14]{Ax01}):
\begin{eqnarray}\label{koltwo}
|u|^2=\frac1{(m+1)(m+3)\dots(m+2j-1)}\sum_{|\al|=j} u_\al^2\al!.
\end{eqnarray}
This proves the lemma in the case $k=0$. Thus we have a base for
the induction on $k$. Let $\vf$ be a smooth function on $\bbR$ and
$v\in\cP_j$. Using the equalities
\begin{eqnarray*}
\De\vf(\vk)=4\vf''(\vk)\vk+2(m+1)\vf'(\vk),\\
\scal{\nabla\vf(\vk)}{\nabla v}=2j\vf'(\vk)v,
\end{eqnarray*}
where $\De$ and $\nabla$ stand for the standard Euclidean
operators on $\bbR^{{m+1}}$, we get
\begin{eqnarray*}
\De(\vf(\vk)v)
=2(2\vf''(\vk)\vk+(m+2j+1)\vf'(\vk))v +\vf(\vk)\De v.
\end{eqnarray*}
If $u\in\cH_j$ and $\vf(\vk)=\vk^k$, then
\begin{eqnarray*}
\De(\vk^k u)=2k\left(m+2j+2k-1\right)\vk^{k-1}u.
\end{eqnarray*}
By (\ref{sympro}),
\begin{eqnarray*}
\wtd{|\vk^ku|}^2=\wtd{\scal{\vk^k u}{\vk^k
u}}=\wtd{\scal{\vk^{k-1}u}{\De(\vk^k u)}}=
2k\left(m+2(j+k)-1\right)\wtd{|\vk^{k-1}u|}^2.
\end{eqnarray*}
This verifies the step of the induction and concludes the proof.
\end{proof}
For the convenience of the reader, we collect the formulas for
coefficients in the following proposition. The definitions of the parameters
$c,s$ are given in (\ref{defcevp}), (\ref{defss}), respectively. The index $j$
corresponds to the summand  $|x|^{n-j}\cH_{j}$ in the decomposition
(\ref{harmdec}). The coefficients $\nu_{j}$ and  the rescaling factors $\tau_{j}$
are defined in Lemma~\ref{snusj} and by the formula (\ref{nnorms}), respectively.
Further, $|\ |$ is the norm of $L^{2}(S^{m})$. The tilde distinguishes objects
corresponding to the Kostlan--Shub--Smale model, in particular,   $\wtd{|\ |}$
is defined by (\ref{kosnorm}).

\begin{proposition}\label{kss-l2}
We have
\begin{eqnarray}
c^2=\dim\cP_n=\sum_{j\in J_n}c_j^2={{m+n}\choose{m}}
\label{csqua},\\
c_j^2=\dim\cH_{j}=\frac{(m+j-2)!(m+2j-1)}{(m-1)!j!},
\label{cksqua}\\
s_j^2=\frac{j(m+j-1)}{m}\label{snmtk},\\
s^2=\frac{1}{c^2} \sum_{j\in J_n}c_j^2s_j^2
=\frac{n(m+n+1)}{m+2}\label{sjsqu},
\end{eqnarray}
and $\nu_j=\frac{c_j^2}{c^2}$. Set $K_n=2^{-n}\Ga\left(\frac{m+1}{2}\right)$.
The coefficients
$\tau_j$ for $|\ |$ and $\wtd{|\ |}$ are subject to the
formula
\begin{eqnarray}\label{ftauj}
\tau_j=\frac{K_n}{\Ga\left(\frac{n-j+2}{2}\right)
\Ga\left(\frac{m+n+j+1}{2}\right)}.
\end{eqnarray}
Furthermore, $\td s_j=s_j$,      $\td\nu_j=n!\,\td c_j^2$, where
\begin{eqnarray}
\td c_j^2=\tau_j c_j^2=
\,\frac{K_n(m+j-2)!(m+2j-1)}
{(m-1)!j!\,\Ga\left(\frac{n-j+2}{2}\right)
\Ga\left(\frac{m+n+j+1}{2}\right)},\nonumber\\
\td c^2=\frac{1}{n!},\label{tdc}\\
\td s^2=n\label{tds}.
\end{eqnarray}
\end{proposition}
\begin{proof}
For any invariant finite dimensional subspace of $L^2(M)$ the
squared norm of the evaluation functional at a point of $M$ is
equal to the dimension of the space since
the integral operator with the kernel
$\phi(p,q)=\scal{\ev_{\cE}(p)}{\ev_{\cE}(q)}$ is the orthogonal
projection onto $\cE$ and its trace is equal to
$\int_M\phi(p,p)\,dp$. This proves the first equalities in
(\ref{csqua}) and (\ref{cksqua}) and, together with
(\ref{dimcpn}), (\ref{dimhn}), and (\ref{cjcjsq}) the remainder of
(\ref{csqua}) and (\ref{cksqua}).

Since $\cH_j$ is the $\la_j$-eigenspace of $-\De_{{S^{m}}}$ with
$\la_j=j(m+j-1)$, (\ref{snmtk}) is a consequence of
Lemma~\ref{eqisct}.

According to (\ref{cksqua}) and (\ref{snmtk}), (\ref{sjsqu}) is
equivalent to the equality
\begin{eqnarray*}
\sum_{j\in J_n}\frac{(m+j-2)!(m+2j-1)}{(m-1)!j!} \cdot
\frac{j(m+j-1)}{m}= {{m+n}\choose{m}}\cdot \frac{n(m+n+1)}{m+2},
\end{eqnarray*}
where $m\geq2$. Assuming ${n\choose k}=0$ for integer $k<0$ or
$k>n$ (this is true for the analytic extension of ${n\choose k}$
on $k$), we may rewrite this equality as
\begin{eqnarray*}
\sum_{j\in J_n}
\left\{{m+j\choose{m+1}}+{m+j-1\choose{m+1}}\right\}=
{m+n+1\choose{m+2}}.
\end{eqnarray*}
The formula above is equivalent to the following one, which is an
easy consequence of the Pascal Triangle equality:
\begin{eqnarray*}
\sum_{j=1}^n{{m+j}\choose{m+1}}={m+n+1\choose{m+2}}.
\end{eqnarray*}

By definition, $\nu_j=\frac{c_j^2}{c^2}$ (see Lemma~\ref{snusj}).
The equality (\ref{ftauj}) follows from  Lemma~\ref{taukj} with
$k=\frac{n-j}{2}$.  Since $\SO(m+1)$ is irreducible in $\cH_j$,
$\td c_j^2=\tau_j c_j^2$ and $\td s_j=s_j$ by Lemma~\ref{eqisct}.

By (\ref{decph}),
$\wtd{\scal{u}{x_0^n}}=n!u(o)$ for all monomials $u\in\cP_n$, where
$o=(1,0,\dots,0)$. Hence $\td\phi(x)=\frac1{n!}x_0^n$,
$\wtd{|\td\phi|}^2=\frac1{n!}$ and (\ref{tdc}) follows.

Due to (\ref{xiphi}), we may compute $s$ applying to $\td \phi_o$
any vector field $\xi\in\so(m+1)$ such that $\xi(o)\neq 0$. Let
$\xi=x_1\frac{\partial}{\partial x_0}-x_0\frac{\partial}{\partial
x_1}$. Then $|\xi(o)|=1$ and we get
\begin{eqnarray}\label{prtds}
\td s^2=\frac{\wtd{|\xi\td\phi_o|}^2}{\td c^2}
=\frac1{n!}\wtd{\big|nx_0^{n-1}x_1\big|}^2 =n.
\end{eqnarray}
This concludes the proof of the proposition.
\end{proof}
Equivalent forms of the equalities (\ref{rescal}) and (\ref{tds}) are known
from Kostlan's papers \cite{Ko93} and  \cite{Ko02}, respectively.
In \cite{Po99}, Podkorytov stated
the equality $s^2=\frac{n(n+m+1)}{m+2}$ without proof.


\section{The scaling limit of the coefficients $\nu_{j}$ as $n\to\infty$}
Now, our aim is to find the scaling limit of the coefficients
$\td\nu_j=\frac{\td c^2_j}{\td c^2}$. The coefficients  $c_j^2$
and $\tau_j$ (see (\ref{csqua}) and (\ref{ftauj}), respectively)
admit the evident extensions on $j$ onto $\bbC$, which are entire
functions which we denote as $c^2(\ze)$ and $\tau(\ze)$.
The function $c^2(\ze)$ is a polynomial of degree $m-1$.
Thus,
\begin{eqnarray*}
\td c^2(\ze)=\tau(\ze) c^2(\ze),\\
\td\nu(\ze)=n!\,\td c^2(\ze)\phantom{,}
\end{eqnarray*}
are entire functions. In this section, $c^2$ and $\td c^2$ denote the
extensions (thus they are not the sums of $c_j^2$ and $\td c_j^2$
as in the previous one). These functions depend on $m$ and $n$
which we omit in the notation.
Note that $\tau$, $c^2$, and $\td\nu$ are positive on the interval
$(0,n)$.
\begin{lemma}\label{tdcconc}
The function $\ln \td\nu$ is strictly concave on the interval
$(0,n)$ and has the unique maximum on it.
\end{lemma}
\begin{proof}
We have $\ln c^2(x)''<0$ because $c^2$ is a product of linear
functions, for $\ln\tau$ the same is true since $(\ln\Ga(x))''>0$.
Hence $\ln\td\nu$ is strictly concave.

Suppose that
\begin{eqnarray}\label{tdceq}
\td\nu(x+2)=\td\nu(x)~~\mbox{for some}~x\in(0,n-2).
\end{eqnarray}
Then $\td\nu$ has a critical point
\begin{eqnarray*}
x_c\in(x,x+2)
\end{eqnarray*}
which necessarily is unique and corresponds to the global maximum
on $(0,n)$. 
Thus, it is sufficient to prove (\ref{tdceq}). Set
\begin{eqnarray*}
\rho_n(x)=\frac{\td\nu(x+2)}{\td\nu(x)},
\end{eqnarray*}
where $x$ runs over $(0,n-2)$.  The condition (\ref{tdceq}) is
equivalent to
\begin{eqnarray}\label{rhoton}
\rho_n(x)=1~~\mbox{for some}~x\in(0,n-2).
\end{eqnarray}
According to (\ref{cksqua}) and (\ref{ftauj}), $\frac{\tau(x+2)}{\tau(x)}
=\frac{n-x}{n+x+m+1}$ and
\begin{eqnarray}
\label{rhofo}
\phantom{xxxx} \rho_n(x)=\left(1+\frac{m-2}{x+1}\right)
\left(1+\frac{m-2}{x+2}\right)
\left(1+\frac{2}{x+\frac{m-1}{2}}\right)\frac{n-x}{n+x+m+1}.
\end{eqnarray}
Due to (\ref{assmn}), $\rho_n(0)=\frac{2mn}{m+n+1}>1$. Assuming
$n>3$ and applying (\ref{assmn}) again, we get
\begin{eqnarray*}
\rho_n(n-2)=\frac{2(n+m-3)(n+m-2)}{n(n-1)(2n+m-5)}
<\frac{4(n+m-3)}{n(2n+m-5)}\leq \frac{n+m-3}{2n+m-5}<1.
\end{eqnarray*}
If $n=3$, then $m=2$ and $\rho_n(n-2)=\frac23$.
Since $n\geq3$ by (\ref{assmn}), this proves (\ref{rhoton}) and
consequently the lemma.
\end{proof}
Set
\begin{eqnarray}\label{mundew}
\mu_n=\sqrt{(m-1)n},\\
\label{defka} \bar\nu_n=\td\nu(x_c),
\end{eqnarray}
where $x_c$ is the critical point of $\td\nu$ (hence $\bar\nu_n$
is the maximum of $\td\nu(x)$ on $[0,n]$).

In the following theorem, we assume that $\td\nu$ is a function on
$(0,\infty)$ which vanishes outside $(0,n)$.
\begin{theorem}\label{limitnu}
If $n$ is sufficiently large, then
\begin{eqnarray}\label{betwe}
\mu_n-\frac{m+1}{2}<x_c<\mu_n+2.
\end{eqnarray}
For any $t>0$
\begin{eqnarray}\label{limnumu}
\lim_{n\to\infty}\frac{\td\nu(\mu_nt)}{\bar\nu_n} =
\left(t^2e^{1-t^2}\right)^{\frac{m-1}{2}},
\end{eqnarray}
where the sequence on the left converges uniformly on
$(0,\infty)$. Moreover,
\begin{eqnarray}\label{asynun}
\bar\nu_n=\frac{A_m}{\sqrt{n}}(1+o(1)),
\end{eqnarray}
as $n\to\infty$, where
$A_m=\frac{2\sqrt{2}}{\Ga\left(\frac{{m}}{2}\right)}
\left(\frac{m-1}{2e}\right)^{\frac{m-1}{2}}$.
\end{theorem}
\begin{proof}
Setting $x=a\sqrt{n}$ in factors of $\rho_n$,  by the straightforward computation
we get
\begin{eqnarray*}
\rho_n(a\sqrt{n})=1+\left(\frac{2(m-1)}{a}-2a\right)\frac1{\sqrt{n}}
+O\left(\frac1n\right).
\end{eqnarray*}
Therefore,
\begin{eqnarray*}
\rho_n(a\sqrt{n})=1+O\left(\frac1n\right)~~~\Longleftrightarrow~~~
a=\sqrt{m-1}.
\end{eqnarray*}
Thus $\mu_n$ can be considered as an approximation of $x_c$.
According to the calculation above,
\begin{eqnarray}\label{muas}
\rho_n(\mu_n)=1-\frac{2(m+1)}{n}+ O\left(n^{-\frac32}\right).
\end{eqnarray}
It follows that
\begin{eqnarray}\label{rhomuo}
\rho_n(\mu_n)<1
\end{eqnarray}
for sufficiently large $n$.
The function $\rho_n$ is positive, decreasing, and convex since
each factor in (\ref{rhofo}) possesses these properties (the
factors are convex because all of them may be written as
$\pm1+\frac{a}{x+b}$, where $a>0$).
Thus (\ref{rhomuo}) gives an upper bound for $x_c$:
\begin{eqnarray*}\label{xclemun}
x_c<\mu_n+2.
\end{eqnarray*}
To find a lower bound, we do one step of the Newton method:
\begin{eqnarray*}
\eta_n=\mu_n-\frac{\rho_n(\mu_n)-1}{\rho_n'(\mu_n)}.
\end{eqnarray*}
Since $\rho_n$ is convex and decreasing, $\rho_n(\eta_n)>1$. Hence
$x_c>\eta_n$. We shall replace $\eta_n$ with a simpler term which
is asymptotic to it as $n\to\infty$. Due to (\ref{muas}), it is
sufficient to do this for $\rho_n'(\mu_n)$. Let us  consider
$(\ln\rho_n)'$. For the first three factors in (\ref{rhofo}) we
may use the equality
$\ln\left(1+\frac{a}{x+b}\right)'=-\frac{a}{(x+a)(x+a+b)}
=-\frac{a}{x^2}+O\left(x^{-3}\right)$ as $x\to\infty$. The
logarithmic derivative of the forth one is equal to
$\frac{2n+m+1}{(n-x)(n+x+m+1)}$. Setting $x=\mu_n$, we get
\begin{eqnarray*}
\frac{\rho_n'(\mu_n)}{\rho_n(\mu_n)}=-\left(2\cdot\frac{m-2}{\mu_n^2}
+\frac{2}{\mu_n^2}+\frac{2}{n}\right)+O\left(n^{-\frac32}\right)
=-\frac{4}{n}+O\left(n^{-\frac32}\right).
\end{eqnarray*}
Together with (\ref{mundew}) and (\ref{muas}), this implies
\begin{eqnarray}\label{rhopmn}
\rho_n'(\mu_n)=-\frac4n+O\left(n^{-\frac32}\right).
\end{eqnarray}
Therefore,
\begin{eqnarray*}
\mu_n-\eta_n=\frac{\rho_n(\mu_n)-1}{\rho_n'(\mu_n)}=\frac{m+1}{2}
+O\left(n^{-\frac32}\right).
\end{eqnarray*}
Furthermore, 
\begin{eqnarray*}
\lim_{n\to\infty}
n^{\frac32}\left(\rho_n\left(\mu_n-\frac{m+1}{2}\right)-1\right)
=\frac{m(m+1)}{\sqrt{m-1}}.
\end{eqnarray*}
Hence $\rho_n\left(\mu_n-\frac{m+1}{2}\right)>1$ for sufficiently
large $n$. This proves (\ref{betwe}).

The following equalities can be proved by a computation:
\begin{eqnarray}
\lim_{n\to\infty}\rho_n(\mu_nx)=1,\label{limrmux}\\
\lim_{n\to\infty} \mu_n\ln\rho_n(\mu_n
x)=2(m-1)\left(\frac1x-x\right),\nonumber
\end{eqnarray}
and the convergence is locally uniform on $(0,\infty)$. The
sequence of piecewise constant functions defined by the equality
\begin{eqnarray}\label{deffn}
f_n(\xi)=\mu_n\ln\rho_n(i)
\end{eqnarray}
for $\xi\in\left(\frac{i-2}{\mu_n},\frac{i}{\mu_n}\right]$, where
$i\in J_n$, converges to $2(m-1)\left(\frac1\xi-\xi\right)$
locally uniformly on $(1,t)$.  Thus,
\begin{eqnarray*}
\int_1^tf_n(\xi)\,d\xi=\frac2{\mu_n}\sum_{i\in J_n,\atop i<\mu_n
t}
\mu_n\ln\rho_n(i)+O\left(n^{-\frac12}\right)= 2\ln \prod_{i\in
J_n,\atop i<\mu_n t}\rho_n(i)+O\left(n^{-\frac12}\right),\\
\lim_{n\to\infty}\ln\frac{\td\nu(\mu_nt)}{\td\nu(\mu_n)}
=\lim_{n\to\infty}\ln \prod_{\mu_n<i<\mu_nt,\atop
n-i~\mbox{\rm\tiny even}}\rho_n(i)
=\frac12\lim_{n\to\infty}\int_1^tf_n(\xi)\,d\xi
\\
=(m-1)\int_1^t\left(\frac1\xi-\xi\right)\,d\xi=\frac{m-1}2(1+2\ln
t-t^2).
\end{eqnarray*}
Here and in what follows, we use Lebesgue's Dominated Convergence
Theorem with the majorant of the type $Ke^{-\eta t}$. The
convergence is locally uniform. Due to (\ref{betwe}),
\begin{eqnarray}\label{barmunu}
\lim_{n\to\infty}\frac{\td\nu(\mu_n)}{\bar\nu_n}
=\lim_{n\to\infty}\frac{\td\nu(\mu_n)}{\td\nu(x_c)}=1.
\end{eqnarray}
Therefore, $\lim_{n\to\infty}\frac{\td\nu(\mu_nt)}{\bar\nu_n}
=\left(t^2e^{1-t^2}\right)^{\frac{m-1}{2}}$ for any $t>1$. The
arguments above with minor changes can be extended onto the case
$0<t\leq1$. Thus (\ref{limnumu}) holds for all $t>0$.  The
sequence $\frac{\td\nu(\mu_nt)}{\bar\nu_n}$ converges uniformly on
any compact interval in $(0,\infty)$. Since $\ln\td\nu(\mu_nt)$ is
concave and $\td\nu(\mu_nt)$ is positive and has maximum near $1$,
this implies the uniform convergence on $(0,\infty)$.

Furthermore,
\begin{eqnarray*}
\lim_{n\to\infty}\frac{2}{\mu_n}\sum_{j\in J_n}
\left(\frac{j^2}{\mu_n^2}
e^{1-\frac{j^2}{\mu_n^2}}\right)^{\frac{m-1}{2}} =\int_0^\infty
\left(t^2e^{1-t^2}\right)^{\frac{m-1}{2}}\,dt
=\frac{\Ga\left(\frac{m}{2}\right)}{2\sqrt{e}}
\left(\frac{2e}{m-1}\right)^{\frac{m}{2}}.
\end{eqnarray*}
On the other hand, the equality $\sum_{j\in J_{n}}\td\nu_j=1$ which is true by
(\ref{sumnul}),
(\ref{limnumu}), and Lebesgue's Dominated Convergence Theorem imply
\begin{eqnarray*}
\lim_{n\to\infty}\bar\nu_n\sum_{i\in J_n}\left(\frac{j^2}{\mu_n^2}
e^{1-\frac{j^2}{\mu_n^2}}\right)^{\frac{m-1}{2}}=1.
\end{eqnarray*}
Computing the ratio of the left-hand parts of the equalities
above, we get
\begin{eqnarray} \label{limnumun}
\lim_{n\to\infty}\bar\nu_n\mu_n=
\frac{4\sqrt{e}}{\Ga\left(\frac{m}{2}\right)}
\left(\frac{m-1}{2e}\right)^{\frac{m}{2}}=A_m\sqrt{m-1}.
\end{eqnarray}
This proves (\ref{asynun}) and the theorem.
\end{proof}
Actually, $\mu_n-\frac{m-1}{2}$ is a better approximation for
$x_c$ than $\mu_{n}$ or $\eta_{n}$ since it is the center of the
interval $(\eta_n,\eta_n+2)$ and $\eta_n$ is close to the solution
of the equation $\rho_n(x)=1$.

The constant $A_m$ in (\ref{asynun}) decreases when $m$ grows and
$\lim_{m\to\infty}A_m=\frac{2}{\sqrt{\pi}}$.
Since $m\geq2$, this implies the inequalities
\begin{eqnarray}\label{ineqam}
\frac{2}{\sqrt{\pi}}<A_m\leq\frac{2}{\sqrt{e}}
\end{eqnarray}
We omit the proof which is standard.

Modifying the arguments above slightly, it is possible to find an
upper bound for the ratio $\frac{\td\nu(j)}{\td\nu(\mu_n)}$.
\begin{proposition}\label{nuabove}
Set $j_n=\min\{j\in J_n:\,j>\mu_n\}$ and let $j\in J_n$, $j>j_n$.
Then
\begin{eqnarray}\label{estnuab}
\frac{\td\nu(j+2)}{\td\nu(j_n+2)}
<\frac{j^{m-1}}{j_n^{m-1}}e^{\frac{j_n^2-j^2}{2n}}.
\end{eqnarray}
\end{proposition}
\begin{proof}
Since $\rho_n(x)$ decreases on $(0,n)$,
\begin{eqnarray}\label{lnrhole}
\int_{j_n}^{j}\ln\rho_n(x)\,dx>2\sum_{i\in J_n,\atop i<
j}\ln\rho_n(i+2)= 2\ln\frac{\td\nu(j+2)}{\td\nu(j_n+2)}.
\end{eqnarray}
The inequality $\ln(1+x)<x$ and (\ref{rhofo}) imply
\begin{eqnarray*}
\ln\rho_n(x)<\frac{m-2}{x+1}+\frac{m-2}{x+2}+\frac{4}{2x+m-1}
-\ln\frac{n+x+m+1}{n-x}.
\end{eqnarray*}
Due to the evident inequality
$\frac{m-2}{x+1}+\frac{m-2}{x+2}+\frac{2}{x+\frac{m-1}{2}}
<\frac{2(m-1)}{x}$, for the integral of the sum of the first three
terms we have the upper bound
\begin{eqnarray*}
2(m-1)\int_{j_n}^{j}\frac{dx}{x} =2(m-1)\ln\frac{j}{j_n}.
\end{eqnarray*}
If $0<t<1$, then $\ln\frac{1+t}{1-t}>2t$. Setting $t=\frac{x}{n}$,
we get the inequality $\ln\frac{n+x+m+1}{n-x}>\frac{2x}{n}$.
Therefore,
\begin{eqnarray}\label{estrtau}
\int_{j_n}^j\ln\frac{n+x+m+1}{n-x}\,dx>\frac{j^2-j_n^2}{n}.
\end{eqnarray}
Thus,
$\int_{j_n}^j\ln\rho_n(x)\,dx<2(m-1)\ln\frac{j}{j_n}+\frac{j_n^2-j^2}{n}$.
Together with (\ref{lnrhole}), this implies (\ref{estnuab}).
\end{proof}

\section{Approximation by polynomials of lower degree}
Let $x$ be the Kostlan--Shub--Smale random polynomial in $\cP_n$.
In this section, we show that $x$ admits a good approximation in
Sobolev spaces on $S^m$ by polynomials of degree $l_n\sim
C\sqrt{n\ln n}$, where $C$ depends on $m$ and on the order of
the Sobolev space.
We estimate the expectation of
$\frac{\dist_X(x,\cP_{l_n})}{\|x\|_X}$ and the probability of the
inequality $\dist_X(x,\cP_{l_n})<\ep\|x\|_X$, where $X$ is the
Sobolev space $H^q=H^q(S^m)$,  $q\geq0$,
with the norm
\begin{eqnarray*}
|x|_q=\Big(|x_0|^2+\sum_{j=1}^\infty j^{2q}|x_j|^2\Big)^{\frac12}.
\end{eqnarray*}
If $j>m-1$, then $j^2<\la_j<2j^2$, where $\la_j=j(j+m-1)$ is the
$j$th eigenvalue of the Laplace--Beltrami operator $\De_{S^{m}}$
on $S^m$. Hence $|x|_q$ is equivalent to the norm
\begin{eqnarray*}
\Big(|x_0|^2+|(-\De_{S^{m}})^{\frac{q}{2}}u|^2\Big)^{\frac12}
=\Big(|x_0|^2+\sum_{j=0}^\infty\la_j^q|x_j|^2\Big)^{\frac12}
\end{eqnarray*}
Clearly, $H^0=L^2(S^m)$.
\begin{lemma}\label{jqtau}
For sufficiently large $n$ the function $\al(x)=x^{2q}\tau(x)$
strictly decreases on the interval $(\sqrt{2qn}+2,n)$.
\end{lemma}
\begin{proof}
Since $\ln\Ga$ is strictly convex, the function $\ln\tau(x)$  is
strictly concave on $(0,n)$. Hence the same is true for
$\ln\al(x)$.  Set
\begin{eqnarray*}
\vf(x)=\frac{\al(x+2)}{\al(x)}
=\left(1+\frac{2}{x}\right)^{2q}\frac{n-x}{n+x+m+1}.
\end{eqnarray*}

Let $q>0$. Then $\vf(x)\to\infty$ as $x\to0$. Hence the inequality
$\vf(a)<1$ for $a\in(0,n)$ implies $\vf(b)=1$ for some
$b\in(0,a)$. Then $\al$ has a critical point in
$(b,b+2)$. Thus, $\al$ decreases on $(a+2,n)$ if $\vf(a)<1$.
A calculation shows that
\begin{eqnarray*}
\vf(\sqrt{2qn})=1-\left(m+1+\frac{1}{q}\right)\frac{1}{n}
+O\left(n^{-\frac32}\right).
\end{eqnarray*}
Hence $\vf(\sqrt{2qn})<1$ for large $n$.

If $q=0$, then $\al(x)=\tau(x)$. Since
$(\ln\tau)'(x)=\frac12\Psi\left(\frac{n-x+2}{2}\right)
-\frac12\Psi\left(\frac{n+x+m+1}{2}\right)$, where
$\Psi=(\ln\Ga)'$, and $(\ln\Ga)''(x)>0$, we get $\tau'(x)<0$ for
all $x>0$. This concludes the proof of the lemma.
\end{proof}

\begin{lemma}
Let $x=x_1\oplus x_2\oplus x_3$ correspond to the decomposition
$\bbR^d=\bbR^{d_1}\oplus\bbR^{d_2}\oplus\bbR^{d_3}$, $a>-d_1$, and
$d_2>2$. Then
\begin{eqnarray}\label{doned}
\int_{S^{d-1}}|x_1|^a\,dx=
\frac{\Ga\left(\frac{d_1+a}{2}\right)\Ga\left(\frac{d}{2}\right)}
{\Ga\left(\frac{d+a}{2}\right)\Ga\left(\frac{d_1}{2}\right)},\\
\label{frain}\int_{S^{d-1}}\frac{|x_1|^2}{|x_2|^2}\,dx=
\frac{d_1}{d_2-2}.
\end{eqnarray}
\end{lemma}
We omit the standard proof.
\begin{theorem}\label{lower}
Let $l_n$ be a sequence of positive integers such that $l_n<n$,
\begin{eqnarray}
\limsup_{n\to\infty}\frac{l_n}{n}<1,\label{boulnn}\\
\lim_{n\to\infty}n^{m+2q}e^{-\frac{l_n^2}{n}}=0.\label{lngrow}
\end{eqnarray}
Suppose that $t_n>0$ satisfy the conditions
\begin{eqnarray}\label{condt}
\lim_{n\to\infty}n^mt_n^{-4}
=\lim_{n\to\infty}t_n^4n^{2q}e^{-\frac{l_n^2}{n}}=0.
\end{eqnarray}
Then there exist $A,B>0$, where $A$ depends only on $q$ and $B$
depends only on the sequence $l_n$, such that for
\begin{eqnarray*}
\eta_n=An^{\frac{m}{2}}t_n^{-2},\\
\ep_n=Bt_nn^{\frac{q}{2}}e^{-\frac{l_n^2}{4n}}
\end{eqnarray*}
we have $\lim_{n\to\infty}\ep_n=\lim_{n\to\infty}\eta_n=0$ and for
any sufficiently large $n$ the inequality
\begin{eqnarray}\label{distvx}
\dist_{H^q}(x,\cP_{l_n})<\ep_n|x|_{q},
\end{eqnarray}
holds for the random Kostlan--Shub--Smale polynomial $x\in\cP_n$
with the probability greater than $1-\eta_n$.
\end{theorem}
\begin{proof}
First of all, we notice that (\ref{condt}) implies
$\ep_n,\eta_n\to0$ as $n\to\infty$.

Let $J_n$ be defined by (\ref{defjnz}). We consider the partition
of $J_n$ by the following subsets of $[0,n]$:
$I_1=[\sqrt{2qn},a\sqrt{n}]$, where $a>\sqrt{2q}$, $I_3=[l_n,n]$,
where $l_n>a\sqrt{n}$, and $I_2=[0,n]\setminus(I_1\mathop\cup I_3)$.
Set
\begin{eqnarray}\label{decuzv}
\cP_n=\cU_n\oplus\cZ_n\oplus\cV_n,
\end{eqnarray}
where each summand is the sum of the spaces $\cH_j$ with $j$
running over the intersection of $J_n$ with one of the sets
$I_1,I_2,I_3$, respectively. For $x\in\cP_n$, let $x=u+z+v$ be the
corresponding decomposition of $x$. Then
\begin{eqnarray*}
|v|_q=\dist_{H^q}(x,\cP_{l_n})
\end{eqnarray*}
since the decomposition (\ref{harmdec}) is orthogonal in the
involved inner products. We shall estimate the probability of the
inequality $|v|<\ep_n|x|$ reducing it to the inequality
$\wtd{|v|}< t_n\wtd{|u|}$ and estimating the expectation of the
ratio $\wtd{|v|}^2/\wtd{|u|}^2$. Since it is homogeneous of degree
$0$, its distribution for the Kostlan--Shub--Smale model and for
the uniform distribution in the unit sphere $\td\cS$ coincide. By
(\ref{frain}),
\begin{eqnarray}\label{expecvu}
\sfE\left(\wtd{|v|}^2/{\wtd{|u|}^{2}}\right)
=\frac{\dim\cV_n}{\dim\,\cU_n-2}.
\end{eqnarray}
According to the decomposition (\ref{harmdec}),
\begin{eqnarray*}
\dim\,\cU_n\geq\dim\cP_{[a\sqrt{n}]-1}-\dim\cP_{[\sqrt{2qn}]}.
\end{eqnarray*}
Replacing $\cV_n$ with $\cP_n$ in
(\ref{expecvu}), we get
\begin{eqnarray*}
\limsup_{n\to\infty}n^{-\frac{m}{2}}
\sfE\left({\wtd{|v|}^2}/{\wtd{|u|}^{2}}\right)\leq
\lim_{n\to\infty}\frac{\dim\cP_n}{n^{\frac{m}{2}}\dim\,\cU_n}
=\frac{1}{a^m-(2q)^{\frac{m}{2}}}.
\end{eqnarray*}
Set
\begin{eqnarray*}
\td V_{t,l}=\{x\in\td\cS:\,\wtd{|v|}>t\wtd{|u|}\}.
\end{eqnarray*}
Let $\td\si$ be the invariant probability measure on $\td\cS$.  Due
to the Chebyshev inequality,
\begin{eqnarray}\label{bgramq}
A>\frac{1}{a^m-(2q)^{\frac{m}{2}}}\kern12pt\Longrightarrow\kern12pt
\td\si\left(\td
V_{t_n,l_n}\right)<An^{\frac{m}{2}}t_n^{-2}\label{tnnmin}
\end{eqnarray}
for all sufficiently large $n$. By (\ref{condt}), $\td\si\left(\td
V_{t_n,l_n}\right)\to0$ as $n\to\infty$. Let
$a>2^{\frac1{m}}\sqrt{2q}$. Then
$a^m-(2q)^{\frac{m}{2}}>(2q)^{\frac{m}{2}}$. Thus for
$A=(2q)^{-\frac{m}{2}}$ the inequality on the right of
(\ref{bgramq}) holds for all sufficiently large $n$.

Let $\al$ be as Lemma~\ref{jqtau}. Then $\al$ decreases on
$(\sqrt{2qn}+2,n)$ and
\begin{eqnarray*}\label{tdtolt}
\begin{array}{r}
|v|_q^2=\sum\limits_{j\in J_n\cap I_3}\al(j)\wtd{|v_j|}^2,\\
|u|_q^2=\sum\limits_{j\in J_n\cap I_1}\al(j)\wtd{|u_j|}^2,
\end{array}
\end{eqnarray*}
where the indices correspond to the decomposition (\ref{harmdec}),
Therefore,
\begin{eqnarray*}\label{vkineq}
\begin{array}{r}
|v|_q^2<\al(l_n)\wtd{|v|}^2,\\
|u|_q^2>\al(a\sqrt{n})\wtd{|u|}^2
\end{array}
\end{eqnarray*}
if $v\neq0$ and $u\neq0$ (we shall assume this in the sequel since
this evidently does not affect the result). Hence $x\notin\td
V_{t_n,l_n}$ implies
\begin{eqnarray}\label{vquqtn}
|v|_q^2
<t_n^2\frac{\al(l_n)}{\al(a\sqrt{n})}|u|_q^2
=t_n^2\left(\frac{l_n^2}{a^2n}\right)^q
\frac{\tau(l_n)}{\tau(a\sqrt{n})}|u|_q^2.
\end{eqnarray}
The ratio $\frac{\tau(l_n)}{\tau(a\sqrt{n})}$ can be estimated as
in Proposition~\ref{nuabove}. The arguments which prove the
inequality (\ref{estrtau}) show that
\begin{eqnarray}\label{lntauaa}
\ln\frac{\tau(l_n+2)}{\tau(a\sqrt{n}+2)}<
{-\int_{a\sqrt{n}}^{l_n}\ln\frac{n+x+m+1}{n-x}\,dx}
<{\frac{a^2}{2}-\frac{l_n^2}{2n}}.
\end{eqnarray}
It follows from the equality $\frac{\tau(x+2)}{\tau(x)}
=\frac{n-x}{n+x+m+1}$  that 
$\lim_{n\to\infty}\frac{\tau(a\sqrt{n}+2)}{\tau(a\sqrt{n})}=1$.
Due  (\ref{boulnn}), for some $b>1$
\begin{eqnarray*}
\limsup_{n\to\infty}\frac{\tau(l_n+2)}{\tau(l_n)}>\frac1b.
\end{eqnarray*}
Hence
\begin{eqnarray}\label{easqtau}
{}\phantom{xxxx}
\limsup_{n\to\infty}\frac{\tau(l_n)e^{\frac{l_n^2}{2n}}}{\tau(a\sqrt{n})}
=\limsup_{n\to\infty}\frac{\tau(l_n+2)e^{\frac{l_n^2}{2n}}}{\tau(a\sqrt{n}+2)}
\frac{\tau(l_n)}{\tau(l_n+2)}\frac{\tau(a\sqrt{n}+2)}{\tau(a\sqrt{n})}
<be^{\frac{a^2}{2}}.
\end{eqnarray}
If $q=0$, then $\left(\frac{l_n^2}{a^2n}\right)^q=1$ and we may
choose arbitrary $a>0$. Thus, the inequality
\begin{eqnarray}\label{vqaquq}
|v|_q^2<Bt_n^2n^qe^{-\frac{l_n^2}{2n}}|u|_q^2
\end{eqnarray}
holds for all sufficiently large $n$ if $B=2b$ and $x\notin\td
V_{t_n,l_n}$ by (\ref{vquqtn}) and (\ref{easqtau}) (we put
$a=\sqrt{2\ln2}$). Let $q>0$. The evident inequality
$\frac{l^2_n}{n}<n$ implies
$\left(\frac{l_n^2}{a^2n}\right)^q<a^{-2q}n^q$. By (\ref{vquqtn})
and the inequalities above, (\ref{vqaquq}) is true provided that
\begin{eqnarray*}
B>a^{-2q}be^{\frac{a^2}{2}}.
\end{eqnarray*}
Set $a=2\sqrt{q}$. Then $a^{-2q}e^{\frac{a^2}{2}}=e^{2q(1-\ln
2-\frac12\ln q)}$. The function on the right attains its maximal
value at $q=\frac{e}{4}$. It is approximately equal to $1.97<2$.
Thus the setting $B=2b$ satisfies (\ref{vqaquq}) for $x\notin\td
V_{t_n,l_n}$ and all $q\geq0$ if $n$ is sufficiently large. This
concludes the proof of the theorem.
\end{proof}
\begin{remark}\rm
The sequence
$t_{n}=n^{\frac{m}{8}-\frac{q}{4}}e^{\frac{l_n^2}{8n}}$ satisfies
(\ref{condt}) if (\ref{lngrow}) is true.
For this choice of $t_n$,
$\eta_n$ is proportional to $\ep_{n}^{2}$. If $l_{n}>\al n$ for
some $\al\in(0,1)$, then both $\ep_{n}$ and $\eta_{n}$ decrease
exponentionally. On the other hand, the expression under the limit
in (\ref{lngrow}) is equal  to $\ep_n^4\eta^2$ up to a
multiplicative constant depending only on $m$ and $q$. Hence
neither $\ep_n^4$ nor $\eta_n^2$ can decay faster than
$n^{m+2q}e^{-\frac{l_n^2}{n}}$.

If $l_n=\sqrt{(m+2q+1)n\ln n}$, then
$n^{m+2q}e^{-\frac{l_n^2}{n}}=\frac1{n}$. For  $t_n=n^a$ we have
$\eta_n=An^{\frac{m}{2}-2a}$, $\ep_n=Bn^{a-\frac{m+1}{4}}$, and
(\ref{condt}) is true if $m<4a<m+1$.
\end{remark}

Let $I$ be an interval in $(0,\infty)$ which may be infinite. Set $I_n=\mu_nI$
and let $\pi_n$ be the orthogonal projection  onto the sum of the spaces $\cH_j$
such that $j\in I_n\cap J_{n}$.
\begin{proposition}\label{approx}
Let $u$ be a random polynomial uniformly distributed in the unit
sphere $\wtd\cS\subseteq\cP_n$  for the norm $\wtd{|\ |}$. Then
\begin{eqnarray}\label{errexp}
\lim_{n\to\infty}\frac{\sfE(|\pi_nu|^2)}{\sfE(|u|^2)}
=\frac1A\int_{I}\left(t^2e^{1-t^2}\right)^{\frac{m-1}{2}}\,dt,
\end{eqnarray}
where $A=\int_0^\infty\left(t^2e^{1-t^2}\right)^{\frac{m-1}{2}}\,dt
=\frac{\Ga\left(\frac{m}{2}\right)}{2\sqrt{e}}
\left(\frac{2e}{m-1}\right)^{\frac{m}{2}}$ and, moreover,
$\sfE(|u|^2)=\frac{m!}{(n+m)!}$.
\end{proposition}
\begin{proof}
Applying (\ref{doned}), (\ref{csqua}), and (\ref{cksqua}), we get
\begin{eqnarray*}
\sfE(|\pi_nu|^2)=\sum\limits_{j\in I_n\cap J_n}\int_{\wtd S}|u_j|^2\,du
=\sum\limits_{j\in I_n\cap J_n}\tau_j\int_{\wtd S}\wtd{|u_j|}^2\,du
=\sum\limits_{j\in I_n\cap J_n}\frac{\tau_jc_j^2}{c^2}
\\
=\frac{\td c^2}{c^2}\sum\limits_{j\in I_n\cap J_n}\td\nu_j.
\end{eqnarray*}
Setting $I=(0,\infty)$, we get
$\sfE\left(|u|^2\right)=\frac{\td c^2}{c^2}=\frac{m!}{(n+m)!}$
according to (\ref{sumnul}),  (\ref{csqua}), and (\ref{tdc}).
As in Theorem~\ref{limitnu}, we may treat the sum as the integral of piecewise
constant functions. Proposition~\ref{nuabove} implies the existence of a common
majorant
for them. This proves the convergence of integrals as $n\to\infty$ and the proposition.
\end{proof}

In particular, for $I=(t,\infty)$ we have the integral
$\int_t^\infty\left(\xi^2e^{1-\xi^2}\right)^{\frac{m-1}{2}}\,d\xi$. Using the elementary
inequality $t^2e^{1-t^2}<e^{-(t-1)^2}$ , which holds if $t>0$ and $t\neq1$,
and the standard estimate of the Gaussian integral, we get the following upper bound
for the expectation of the distance in $L^2(S^m)$ from the random uniformly distributed
 polynomial in $\td\cS\subset\cP_n$ to the space $\cP_{l_n}$ which holds if $t>1$,
$l_n>t\mu_n$, $l_n\in J_n$, and $n$ is sufficiently large:
\begin{eqnarray*}
\sfE(\dist(u,\cP_{l_n})^2)<
\frac{e^{-\frac{m-1}{2}(t-1)^2}}{\sqrt{2e(m-1)}\,(t-1)}
\sfE(|u|^2)
\end{eqnarray*}
(we drop the standard calculation). It $I$ is a neighborhood of $1$, then the
integrand in (\ref{errexp}) decays exponentially outside $I$ as $m$ grows.
Using arguments of Proposition~\ref{nuabove}, it is possible to
prove that  $1-\frac{\sfE(|\pi_nu|^2)}{\sfE(|u|^2)}\leq e^{-Km}$ for sufficiently
large $n$, where $K$ depends only on $I$.


\end{document}